\documentclass[11pt]{article}

\usepackage{latexsym,color,amsmath,amsthm,amssymb,amscd,amsfonts}

\newtheorem{theo}{Theorem}
\newtheorem{lem}[theo]{Lemma}
\newtheorem{cor}[theo]{Corollary}
\newtheorem{prop}[theo]{Proposition}
\newtheorem{defn}[theo]{Definition}
\newtheorem{exam}[theo]{Example}
\def\R{\R}

\def\vphi{{\varphi}}

\def\R{{\mathbb R}}
\def\grad{\nabla}

\def\eps{\varepsilon}

\def\qed{\hfill $\vcenter{\hrule height .3mm
\hbox {\vrule width .3mm height 2.1mm \kern 2mm \vrule width .3mm
height 2.1mm} \hrule height .3mm}$ \bigskip}

\def\eps{\varepsilon}

\def\to{\rightarrow}

\newcommand \supp{\operatorname{supp} \,}
\def\pmx{\begin{pmatrix}}
\def\emx{\end{pmatrix}}
\def\Hess{{\nabla^2}}

\def\det{{\rm det}}
\def\supp{{\hbox{supp}}}
\def\Supp{{\hbox{supp}}}
\def\interior{{\hbox{int}}}

\def\R{\mathbb R}

\def\N{\mathbb{N}}

\begin{document}

\title{Functional affine-isoperimetry and an inverse logarithmic Sobolev inequality
\footnote{Keywords: affine isoperimetric inequality, logarithmic Sobolev inequality. 2010 Mathematics Subject Classification: 52A20,  }}

\author{S. Artstein-Avidan\thanks{Partially supported by BSF grant No. 2006079 and by ISF grant No. 865/07}, B. Klartag\thanks{Partially supported by an ISF grant and an IRG grant} ,   C. Sch\"utt
and E. Werner\thanks{Partially supported by  NSF grant  and  BSF grant No. 2006079}}

\date{}

\maketitle
\begin{abstract}
We give a functional version of the affine isoperimetric inequality
for log-concave functions which may be interpreted as an inverse
form of a logarithmic Sobolev inequality inequality for entropy. A
linearization of this inequality gives an inverse inequality to the
Poincar\'e inequality for the Gaussian measure.
\end{abstract}

\section{Introduction}
There is a general approach to extend invariants of convex bodies to
the corresponding invariants of functions \cite{ArtKlarMil,
Fradelizi+Meyer, Klar, Milman}. We investigate here the
affine surface area and the affine isoperimetric inequality and their
corresponding invariants for log-concave functions. The affine
isoperimetric inequality corresponds to an inequality that may be
viewed as an inverse logarithmic Sobolev inequality   for entropy. A
linearization of this inequality yields an inverse inequality to a
Poincar\'e inequality.

Logarithmic Sobolev inequalities provide upper bounds for the
entropy. There is a vast amount of literature on logarithmic Sobolev
inequalities and related topics, e.g.
 \cite{BarKol, BobLed, Car, Gross, LatOle, LYZ2002, Stam}. We quote only
the sharp logarithmic Sobolev inequality for the Lebesgue measure on
$\mathbb R^{n}$  (see, e.g., \cite{BePe})
\begin{eqnarray}\label{logsob1}
\int_{\operatorname{supp}(f)}|f|^{2}  \ln(|f|^2)dx
-\left(\int_{\R^n}|f|^{2}dx\right)\ln\left(
\int_{\R^n}|f|^{2}dx\right) \leq
\frac{n}{2}\ln\left(\frac{2}{\pi e n}\int_{\R^n}\|\grad
f\|^{2}dx \right),
\end{eqnarray}
with equality if and only if $f(x) = (2 \pi)^{-(n/4)} \exp(-\| x - b
\|^2/4)$ for a vector $b \in \R^n$. Here, and throughout the
paper, $\| \cdot \|$ denotes the standard Euclidean norm  and $\langle \cdot, \cdot\rangle$ denotes the standard scalar product on $\R^n$. This
inequality is directly equivalent to the Gross logarithmic Sobolev
inequality \cite{BePe, Gross}
\begin{equation}\label{logsob2}
\int_{\operatorname{supp}(h)}   |h|^2
\ln\left(\frac{|h|}{\|h\|_{L^2(\gamma_{n})}}\right) d \gamma_{n}
\leq \int_{  \R^n} \|\grad h\|^{2}d\gamma_{n},
\end{equation}
where $\gamma_{n}$ is the normalized Gauss measure on $\mathbb
R^{n}$,
$d\gamma_{n}=(2\pi)^{-\frac{n}{2}}e^{-\frac{\|x\|^{2}}{2}}dx$. Equation (\ref{logsob2})
becomes an equality if and only if $h(x)=c e^{\langle a, x \rangle}$ with
$c > 0$ and $a \in \R^n$.

We will now integrate by parts, and rewrite the logarithmic Sobolev inequality as an upper bound
for the entropy in terms of the Laplacian of the function. The main result in this note shall be
a lower bound for  entropy in terms of the Laplacian, the difference between the two bounds being an interchange between integration and
logarithm and  replacement of the arithmetric mean of the eigenvalues of
the Hessian by the geometric mean.

 We shall need some more notation.
Let $(X, \mu)$ be a measure space and let $f: X \rightarrow
\R$ be a measurable function.  Denote the support of $f$ by
$\operatorname{supp}(f)= \{x: f(x) \neq 0\}$. Then the {\em entropy}
of $f$,  $\operatorname{Ent}(f) $,  is defined (whenever it makes
sense) by
\begin{equation}\label{entropy}
\operatorname{Ent}(f)
=  \int_{\operatorname{supp}(f)} |f| \ln(|f|) d \mu - \|f\|_{L^1(X, \mu)} \ln \|f\|_{L^1(X, \mu)} =  \int_{\operatorname{supp}(f)} f \ln\left(\frac{|f|}{\|f\|_{L^1(X, \mu)}}\right) d \mu,
\end{equation}
where $\|f\|_{L^1(X, \mu)} = \|f\|_{L^1( \mu)}= \int_X |f| d \mu$.
In particular, if $\|f\|_{L^1(X, \mu)} = 1$,
$$
\operatorname{Ent}(f) =  \int_{\supp(f)} |f| \ln(|f|) d \mu .
$$

If $f$ is  a positive function,   we get in (\ref{logsob1})
\begin{eqnarray}\label{logsob3}
\operatorname{Ent}(f)
&=&\int_{\R^n}f   \ln(f)dx
-\left(\int_{\R^n}fdx\right)\ln\left(\int_{\R^n}fdx\right)
\nonumber \\
&\leq&  \frac{n}{2}\ln\left(\frac{2}{\pi e n}
\int_{\R^n}\|\grad \sqrt{f}\|^{2}dx \right) =
\frac{n}{2}\ln\left(\frac{1}{2\pi e n}
\int_{\R^n}\frac{\|\grad f\|^{2}}{f}dx \right).
\end{eqnarray}
For a sufficiently smooth function $f$ defined on  $\R^{n}$, we denote the Hessian of
$f$ by $\Hess\left(f\right) = \left(\frac{\partial ^2 f}{\partial x_i \partial x_j} \right)_{i,j=1,\ldots,n}$.
Note that
\begin{equation}\label{Fisher22}
\int_{{\supp (f)}}\frac{\|\grad f\|^{2}}{f}dx =
\int_{\supp (f)} f \ \bigg(\operatorname{tr}
\left(\Hess
 \left(-\ln f\right)\right)\bigg)dx.
\end{equation}
For $f \geq 0$ with $\int f dx=1$, this is the {\em Fisher
information}.
Equation (\ref{Fisher22}) is easily verified using integration by parts.

The logarithmic Sobolev inequality (\ref{logsob3}), together with
(\ref{Fisher22}), becomes
\begin{equation}\label{LogSob55}
\operatorname{Ent}(f)+\ln((2\pi e)^{\frac{n}{2}}) \leq
\frac{n}{2}\ln\left[\frac{1}{n}
 \int_{\supp (f)} f \ \bigg(\operatorname{tr}
\left(\Hess
 \left(-\ln f\right)\right)\bigg)dx\right].
\end{equation}

The main goal in this paper is to prove, for log-concave functions, a
converse of inequality (\ref{LogSob55}). A function $f: \R^n \rightarrow \R$
is called log-concave if it takes
the form $\exp(-\Psi)$ for a convex function $\Psi: \R^n \rightarrow
\R \cup \{ \infty \}$. We shall usually assume also that the function is upper semi-continuous.

This converse log Sobolev inequality
is stated in
 the
following theorem.  It
relates entropy to a new expression, which can be thought of as an
affine invariant  version of Fisher information.

 The inequality is obtained by suitably applying and analysing the affine isoperimetric inequality, which,
 for convex bodies $K$ in $\R^n$,  gives an upper
 bound for the affine surface area.
Affine surface area measures and their related inequalities (see below for the
definition and statements)  have attracted
considerable attention recently e.g. \cite{GaZ, LR2, Lu2, SW2002,
Z3}.

\begin{theo}
Let
$f:\R^{n}\rightarrow [0, \infty)$ be an upper semi-continuous log-concave
function which belongs to $C^2(\Supp(f))  \cap L^1(\R^n, dx)$  and such that
$f \ln f $
and $f  \ln\det\left( \Hess \left(-\ln f\right)\right)\bigg) \in   L^1(\supp (f), dx)$.
Then
\begin{eqnarray*}
 \int_{\supp (f)} f \ \ln \bigg(\det \left( \Hess\left(-\ln f\right)\right)\bigg)  dx
 \leq 2 \bigg[  \operatorname{Ent}(f) +  \|f\|_{L^1( dx)} \ln(2 \pi e)^\frac{n}{2} \bigg].
\end{eqnarray*}
 There is equality for $f(x)=C e^{-\langle A x, x \rangle}$, where
 $C>0$ and $A$ is an $n \times n$ positive-definite matrix of
 determinant one. 
 \label{thm1}
 \end{theo}

It is important to note the affine invariant nature of Theorem
\ref{thm1}. Both the left-hand side and the right-hand side are
invariant under volume-preserving linear transformations. This is
not the case with the logarithmic Sobolev inequality.
 The expression on the right-hand side of
(\ref{LogSob55}) involves the arithmetric mean $\frac{1}{n}\
\left(\operatorname{tr} \left(\Hess
\left(-\ln f\right)\right) \right)$ of the eigenvalues of
$\Hess \left(-\ln f\right)$. The expression on
the left-hand side of Theorem \ref{thm1} can be written as $n \
\ln \bigg(\det\left(
\Hess\left(-\ln f\right)\right)\bigg)^\frac{1}{n}$
and involves the geometric mean of the eigenvalues of
$\Hess  \left(-\ln f\right)$. Thus, we get from
an upper  bound for the entropy to a lower bound for the entropy by
interchanging integration and logarithm and by replacing the
arithmetic mean of the eigenvalues of the Hessian by its geometric
mean.

As  the entropy  for the Gaussian random variable
$g(x)=\frac{1}{(2\pi)^\frac{n}{2}}e^{-\frac{\|x\|^{2}}{2}}$ is
$\operatorname{Ent}(g) = - \ln(2 \pi e)^\frac{n}{2}$, Theorem \ref{thm1}
 immediately implies the following corollary. \vskip 3mm
\begin{cor} Let $f:\R^{n}\rightarrow\mathbb [0, \infty)$ be a
log-concave function such that  $ f  \in C^2(\R^n)$, $
\|f\|_{L^1( dx)} =1$  and such that $f \ln f $ and $ f \ln \left(\det \Hess
 \left(-\ln f\right)\right) \in   L^1(\supp (f), dx)$. Then
$$
  \int_{\supp (f)} f \ln \bigg( \det \left( \Hess \left(-\ln f\right)\right)\bigg) dx
\leq
2\big(
\operatorname{Ent}(f)
-\operatorname{Ent}(g)\big),
$$
with equality for $f(x)=e^{-\pi \langle Ax, x \rangle}$ for a
positive-definite matrix $A$ of determinant one. \label{cor_2}
\end{cor}
The  expression  $\operatorname{Ent}(f)
-\operatorname{Ent}(g)$ is  called the entropy gap.
The linearization of Theorem \ref{thm1}  yields the following
corollary, and alternative proof of which, together with a generalization,
is also given below in Section \ref{lin}.

\begin{cor}\label{CorPoin}
For all functions $\varphi\in C^{2}( \R^{n})\cap
L^{2}(\R^{n}, \gamma_{n})$ with $\|\Hess
\varphi \|_{HS}\in L^{2}(\R^{n},\gamma_{n})$ we have
\begin{equation}\label{CorPoin1_}
\int_{\R^{n}}  \left[ \left\|\grad \varphi\right\|^2_{} - \frac{\|\Hess
\varphi\|_{HS}^2}{2}
\right] d\gamma_n \le \mbox{Var}_{\gamma_n}(\varphi).
\end{equation}
Here,
$\|\ \|_{HS}$ denotes the Hilbert-Schmidt norm and
$\mbox{Var}_{\gamma_n}(\varphi)= \int_{\R^{n}} \varphi^2
d\gamma_n - \left(\int_{\R^{n}} \varphi d\gamma_n\right)^2$
is the variance. There is equality for all polynomials of degree
$2$.
\end{cor}

The Poincar\'e inequality for the Gauss measure
is ( see \cite{Beck1})
$$
\int_{\R^{n}}|f|^{2}d\gamma_{n}
-\left(\int_{\R^{n}}fd\gamma_{n}\right)^{2}
\leq\int_{\R^{n}}\|\grad f\|^{2}d\gamma_{n}.
$$
Hence, the inequality of Corollary  \ref{CorPoin} gives  a reverse
Poincar\'e inequality.
We shall also give an alternative proof of Corollary \ref{CorPoin}, which generalizes to the following family of inequalities (which we state only in the one dimensional case for simplicity)

\begin{theo}\label{gen-inv-poinc}
For all $m$ and all $\varphi \in C^{m,2}(\R)$ 
with
$\int \varphi d\gamma = 0$, one has
\[\int \sum_{j=0}^{m-1} \frac{\left(\varphi^{(2j+1)}\right)^2}{(2j+1)!} d\gamma
\le \int   \sum_{j=0}^m \frac{\left(\varphi^{(2j)}\right)^2}{(2j)!}
d\gamma \le   \int \sum_{j=0}^m
\frac{\left(\varphi^{(2j+1)}\right)^2}{(2j+1)!} d\gamma .\]
\end{theo}
\noindent Here
$\gamma = \gamma_1$ denotes the one-dimensional standard Gaussian distribution, and $C^{m,2}(\R)$ means functions which are $m$ times continuously differentiable whose respective derivatives belong to $L_2$.

Our  results are formulated and
proved for functions that are sufficiently smooth. However, they can
be generalized to functions  that are not necessarily satisfying any
$C^2$-assumptions. We then need to replace the second derivatives by
the generalized second derivatives (compare e.g. \cite{SW2002}).

\medskip

\section{Affine isoperimetry for $s$-concave functions}
\begin{defn} \label{def1}
Let  $s, n \in \N$. We say that $f: {\R}^n\to [0, \infty)$ is  $s$-concave, and denote
$f\in Conc_s({\R}^n)$, if
$f$  is upper semi continuous,
 $\overline{\supp (f)}$
is a  convex body (convex, compact and with non-empty interior) and $f^\frac{1}{s}$ is
concave on $\overline{\supp (f)}$. The class $Conc_s^{(2)}(\R^n)$ shall consist of
such $f\in Conc_s(\R^n)$ which are twice continuously differentiable in the interior of their support.
\end{defn}
Note that for every $f\in Conc_s(\R^n)$ there exists a constant $C>0$ such that
$0\leq f\leq C$. In particular, such an $f$ is integrable.

As  in \cite{ArtKlarMil}, we associate with a function $f \in Conc_s(\R^n)$ a the convex body $K_{s}(f)$
in $\R^{n}\times \R^{s}$ given by
\begin{equation}\label{def.Ksf}
K_s(f):=\big\{(x,y) \in \R^n \times \R^s: x \in
\supp (f), \|y\| \leq
f^\frac{1}{s}(x)\big\}.\end{equation}
A special function in the class $Conc_s(\R^n)$, which will
play the role of the Euclidean ball in convexity, is
 \[g_s(x):=(1-
\|x\|^2)_+^\frac{s}{2} \] where, for $a \in \mathbb{R}$, $a_+ = \max \{a, 0 \}$. It follows
immediately from the definition that $K_s(g_s)=B^{n+s}_2$, the
$(n+s)$-dimensional Euclidean unit ball centred at the origin.
By Fubini's theorem, we have that for all $f \in
Conc_s(\R^n)$
\[ \text{vol}_{n+s}\left(K_s(f)\right) = \text{vol}_s(B_2^s) \int_{\R^n} f dx. \]

An important affine invariant quantity in convex
geometric analysis is  the affine surface area which, for  a convex
body $K\subset \R^n$ with a smooth boundary is defined by
\begin{equation} \label{def:affine}
as_{1}(K)=\int_{\partial K}  \kappa_K(x)^{\frac{1}{n+1}}
 d\mu_{ K}(x).
\end{equation}
 Here, $\kappa (x) =\kappa_{K}(x)$ is the
generalized Gaussian curvature at the   point $x$ in  $\partial K$, the
boundary of $K$, and $\mu=\mu_K$ is the surface area measure on  the
boundary $\partial K$. See e.g. \cite{Lu1991, MW1, SW1990} for extensions of the
definition of affine surface area to an arbitrary convex body in
$\R^n$. For a function $f\in Conc_s(\R^n)$, we define
\begin{equation}\label{def:aff}
as_1^{(s)}(f)=as_1\left(K_s(f) \right).
\end{equation}
Our first goal is to give a precise formula for $as_1^{(s)}(f)$ in
terms of derivatives of the function $f$. This is done in the next
proposition.
There, for $x, y >0$,
$$B(x,y)= \int_0^1 t^{x-1}(1-t)^{y-1} dt$$ is the Beta function.

\begin{prop}\label{prop.s-affine-sa}
Let $s\in \N$ and $f \in Conc_s^{(2)}(\R^n)$. Then
\begin{equation*}
as_1^{(s)}(f)=c_s \int_{\supp (f)}
\left| \det \left( {\Hess f^\frac{1}{s}
} \right) \right|^\frac{1}{n+s+1}
f^{\frac{(s-1)(n+s)}{s(n+s+1)} }\ dx.
\end{equation*}
Here,  $c_s=(s-1) \text{vol}_{n-1}(B^{s-1}_2) B\left(\frac{s-1}{2},\frac{1}{2}\right)$ if $s
\neq 1$ and $c_1=2$.
\end{prop}

In order to derive the formula for $as_1^{(s)}(f)$, we have  to
compute the affine surface area of the body $K_s(f)$. To
this end, we  compute the curvature of this body, which is
circular in $s$ directions, and is behaving like $f^{1/s}$ in the
other directions. We make use of the following well known lemma.

\begin{lem} \label{lem.kappa}(\cite{Tho}, p. 93, exercise 12.13)
Let $h: \R^n \rightarrow [0,+\infty)$ be twice continuously
differentiable. Let $x=(t,h(t))\in \R^{n}\times \R$ be a point on the graph of $h$.  Then,
with the appropriate orientation, the Gauss curvature $\kappa$ at $x$ is
$$
\kappa(x)=\frac{ \det (\Hess h)}{(1 + \|\grad h \|^2)^{\frac{n+2}{2}}}.
$$
\end{lem}
\vskip 3mm

We shall apply Lemma \ref{lem.kappa} to the boundary
of a convex body $K$. We consider only the orientation
that gives nonnegative curvature. Thus, for a point $x \in \partial K$ whose boundary is described locally by the convex function $h$  we can use the
formula
\begin{equation}\label{curv}
\kappa(x)
=\left|
\frac{ \det \left(\Hess h \right)}{(1 + \|\grad h \|^2)^{\frac{n+2}{2}}}\right|.
\end{equation}
We shall denote by
$N_K(x)$ the outer unit normal vector   to $\partial K$ at $x\in \partial K$.

\begin{lem} \label{normal+kappa}
Let $f \in Conc_s^{(2)}(\R^n)$.
 Then for all $x = (x_1,\ldots,x_{n+s}) \in \partial K_s(f)$ with
 $(x_1\ldots,x_n) \in \interior(\Supp(f))$,\\
\noindent (i)    $N_{K_s(f)}(x)=   \frac{\left(
f^\frac{1}{s}\nabla f^\frac{1}{s}, -x_{n+1}, \dots , -x_{n+s}
\right) }{f^\frac{1}{s}\left(1+ \| \nabla
f^\frac{1}{s}\|^2\right)^\frac{1}{2}}$, \\  \noindent (ii)
$\kappa_{K_s(f)}(x)= \left|\frac{ \det \left(\Hess \
f^\frac{1}{s}\right)}{f^\frac{s-1}{s} \left(1 +\|\nabla
f^\frac{1}{s}\|^2\right)^\frac{n+s+1}{2}}\right|.$\\
Here, $f$ is evaluated at $(x_1,\ldots,x_n) \in \R^n$.
\end{lem}

\begin{proof}[Proof of Lemma \ref{normal+kappa}]
If $s=1$, (i) of the lemma follows immediately from elementary calculus and (ii) from Lemma \ref{lem.kappa}.

Therefore, we can assume that $s \geq 2$.
Since,
by equation (\ref{def.Ksf}),  the boundary of $K_s(f)$ is given by
$\{(x, y ) \in \R^n \times \R^s: \|y\| = f^{1/s}(x)\}$, the
boundary of $K_s(f)$ is the union of the graphs of the two mappings
$$(x_1,\dots,x_n,x_{n+1},\dots, x_{n+s-1}) \rightarrow
(x_1,\dots,x_n, x_{n+1},\dots, x_{n+s-1}, \pm x_{n+s}),$$ where, with $x = (x_1, \ldots, x_n)$,
\begin{eqnarray}\label{xn+s}
x_{n+s}&=&\bigg(f^\frac{2}{s}(x)-\sum_{i=n+1}^{n+s-1}x_i^2\bigg)^\frac{1}{2}.
\end{eqnarray}
Because of symmetry, it is enough to consider only the ``positive'' part
of $\partial K_s(f)$, in which the last coordinate is non-negative.
We will show that the outer normal and the curvature exist for $(x,
y)$ with $x \in \supp (f)$ and $\|y\|=f(x)^\frac{1}{s}$ (they
may not exist for $x \in \partial \left(\supp (f)\right)$).

Letting $g=f^{1/s}$ we have
$$
x_{n+s}=\sqrt{g(x_{1},\dots,x_{n})^{2}-\sum_{i=n+1}^{n+s-1}x_{i}^{2}}.
$$
As $f^\frac{1}{s}$ is  everywhere differentiable on its support, we have for $i$ with $1\leq i\leq n$ and, provided $s \geq 2$, for  $j$ with $n+1\leq j\leq n+s-1$,
\begin{eqnarray}\label{diff1}
\frac{\partial x_{n+s}}{\partial x_{i}}
=\frac{g\frac{\partial g}{\partial
x_{i}}}{\sqrt{g^{2}-\sum_{i=n+1}^{n+s-1}x_{i}^{2}}} =\frac{g\frac{\partial
g}{\partial x_{i}}}{x_{n+s}}
\hskip 5mm \text{and} \hskip 5mm
\frac{\partial x_{n+s}}{\partial x_{j}}
=-\frac{x_{j}}{\sqrt{g^{2}-\sum_{i=n+1}^{n+s-1}x_{i}^{2}}}
=-\frac{x_{j}}{x_{n+s}}.
\end{eqnarray}
\vskip 2mm
\noindent
(i) Therefore we get for
almost all $z \in \partial K_s(f)$ with (\ref{xn+s}) and (\ref{diff1})
 \begin{eqnarray*}
\noindent
N_{K_s(f)}(z) = \frac{(\nabla x_{n+s}, -1)}{\left(1+ \| \nabla x_{n+s}\|^2\right)^\frac{1}{2}} = \frac{\left( f^\frac{1}{s}\nabla f^\frac{1}{s}, -x_{n+1}, \dots , -x_{n+s} \right) }{f^\frac{1}{s}\left(1+ \| \nabla f^\frac{1}{s}\|^2\right)^\frac{1}{2}}.
\end{eqnarray*}

\noindent
(ii) We have for all $i$ with $1\leq i\leq n$,
$$
\frac{\partial^{2} x_{n+s}}{\partial x_{i}^{2}}
=\frac{g\frac{\partial^{2}g}{\partial x_{i}^{2}}
+(\frac{\partial g}{\partial x_{i}})^{2}}{x_{n+s}}
-\frac{g^{2}(\frac{\partial g}{\partial x_{i}})^{2}}{x_{n+s}^{3}}
=\frac{g\frac{\partial^{2}g}{\partial x_{i}^{2}}
}{x_{n+s}}
-\frac{(\frac{\partial g}{\partial x_{i}})^{2}
\sum_{j=n+1}^{n+s-1}x_{j}^{2}}{x_{n+s}^{3}}.
$$
For $i\ne j$ with $1\leq i,j\leq n$,
\begin{eqnarray*}
\frac{\partial^{2} x_{n+s}}{\partial x_{i}\partial x_{j}}
&=&\frac{g\frac{\partial^{2} g}{\partial x_{i}\partial
x_{j}}+\frac{\partial g}{\partial x_{i}}\frac{\partial g}{\partial
x_{j}}}{x_{n+s}} -\frac{g^{2}\frac{\partial g}{\partial x_{i}}
\frac{\partial g}{\partial x_{j}}}{x_{n+s}^{3}}
=\frac{g\frac{\partial^{2} g}{\partial x_{i}\partial
x_{j}}}{x_{n+s}} -\frac{\frac{\partial g}{\partial x_{i}}
\frac{\partial g}{\partial x_{j}}
\sum_{j=n+1}^{n+s-1}x_{j}^{2}}{x_{n+s}^{3}}.
\end{eqnarray*}
For $1\leq i\leq n$ and $n+1\leq j\leq n+s-1$
$$
\frac{\partial^{2} x_{n+s}}{\partial x_{i}\partial x_{j}}
=-\frac{x_{j}g\frac{\partial g}{\partial x_{i}}}{x_{n+s}^{3}}.
$$
For $n+1\leq i\leq n+s-1$,
$$
\frac{\partial^{2} x_{n+s}}{\partial x_{i}^{2}}
=-\frac{1}{x_{n+s}}-\frac{x_{i}^{2}}{x_{n+s}^{3}}
=-\frac{x_{n+s}^{2}+x_{i}^{2}}{x_{n+s}^{3}}.
$$
For $i$ and $j$ with $n+1\leq i,j\leq n+s-1$ and $j \neq i$,
$$
\frac{\partial^{2} x_{n+s}}{\partial x_{i}\partial x_{j}}
=-\frac{x_{i}x_{j}}{x_{n+s}^{3}}.
$$
We compute now the determinant of the following $[n+(s-1)]\times [n+(s-1)]$ matrix
$$
\left(\begin{array}{cccccc} \frac{g\frac{\partial^{2}g}{\partial
x_{1}^{2}} }{x_{n+s}} -\frac{(\frac{\partial g}{\partial x_{1}})^{2}
\sum_{j=n+1}^{n+s-1}x_{j}^{2}}{x_{n+s}^{3}} &\dots &
\frac{g\frac{\partial^{2} g}{\partial x_{1}\partial x_{n}}}{x_{n+s}}
-\frac{\frac{\partial g}{\partial x_{1}} \frac{\partial g}{\partial
x_{n}} \sum_{j=n+1}^{n+s-1}x_{j}^{2}}{x_{n+s}^{3}} & \hskip 5mm
-\frac{x_{n+1}g\frac{\partial g}{\partial x_{1}}}{x_{n+s}^{3}}
&\dots &
-\frac{x_{n+s-1}g\frac{\partial g}{\partial x_{1}}}{x_{n+s}^{3}}  \\
\vdots &  &\vdots   &   \vdots  & &\vdots  \\
\frac{g\frac{\partial^{2} g}{\partial x_{n}\partial x_{1}}}{x_{n+s}}
-\frac{\frac{\partial g}{\partial x_{n}} \frac{\partial g}{\partial
x_{1}} \sum_{j=n+1}^{n+s-1}x_{j}^{2}}{x_{n+s}^{3}} &\cdots &
\frac{g\frac{\partial^{2}g}{\partial x_{n}^{2}} }{x_{n+s}}
-\frac{(\frac{\partial g}{\partial x_{n}})^{2}
\sum_{j=n+1}^{n+s-1}x_{j}^{2}}{x_{n+s}^{3}} &\hskip 5mm
-\frac{x_{n+1}g\frac{\partial g}{\partial x_{n}}}{x_{n+s}^{3}}
&\cdots&
-\frac{x_{n+s-1}g\frac{\partial g}{\partial x_{n}}}{x_{n+s}^{3}}  \\
-\frac{x_{n+1}g\frac{\partial g}{\partial x_{1}}}{x_{n+s}^{3}}
&\cdots & -\frac{x_{n+1}g\frac{\partial g}{\partial
x_{n}}}{x_{n+s}^{3}} &-\frac{x_{n+s}^{2}+x_{n+1}^{2}}{x_{n+s}^{3}}
&\cdots
&-\frac{x_{n+s-1}x_{n+1}}{x_{n+s}^{3}} \\
\vdots &   & \vdots&\vdots&  &\vdots      \\
-\frac{x_{n+s-1}g\frac{\partial g}{\partial x_{1}}}{x_{n+s}^{3}} &
\cdots   &
-\frac{x_{n+s-1}g\frac{\partial g}{\partial x_{n}}}{x_{n+s}^{3}} &
-\frac{x_{n+s-1}x_{n+1}}{x_{n+s}^{3}}
&\cdots&-\frac{x_{n+s}^{2}+x_{n+s-1}^{2}}{x_{n+s}^{3}}
\\
\end{array}
\right)
$$
For fixed $i$, $1\leq i\leq n$ we multiply
each of the rows $n+1\leq j\leq n+s-1$ by
$$
\frac{x_{j}}{g}\frac{\partial g}{\partial x_{i}}
$$
and add them up. We obtain the vector
$$
\left(-\frac{\frac{\partial g}{\partial x_{i}}
\frac{\partial g}{\partial x_{1}}
\sum_{j=n+1}^{n+s-1}x_{j}^{2}}{x_{n+s}^{3}},\dots,
-\frac{\frac{\partial g}{\partial x_{i}}
\frac{\partial g}{\partial x_{n}}
\sum_{j=n+1}^{n+s-1}x_{j}^{2}}{x_{n+s}^{3}},
-\frac{x_{n+1}g\frac{\partial g}{\partial x_{i}}}{x_{n+s}^{3}},
\dots,-\frac{x_{n+s-1}g\frac{\partial g}{\partial x_{i}}}{x_{n+s}^{3}}
\right)
$$
and subtract it from the $i$-th row.
The determinant does not change
and we obtain
$$
\left(\begin{array}{cccccc}
\frac{g\frac{\partial^{2}g}{\partial x_{1}^{2}}
}{x_{n+s}}
&\cdots &
\frac{g\frac{\partial^{2} g}{\partial
x_{1}\partial x_{n}}}{x_{n+s}}
& \hskip 5mm
0
&\cdots &
0  \\
\vdots &  &\vdots   &   \vdots  & &\vdots  \\
\frac{g\frac{\partial^{2} g}{\partial
x_{n}\partial x_{1}}}{x_{n+s}}
&\cdots &
\frac{g\frac{\partial^{2}g}{\partial x_{n}^{2}}
}{x_{n+s}}
&\hskip 5mm
0
&\cdots&
0  \\
-\frac{x_{n+1}g\frac{\partial g}{\partial x_{1}}}{x_{n+s}^{3}}
&\cdots & -\frac{x_{n+1}g\frac{\partial g}{\partial
x_{n}}}{x_{n+s}^{3}} &-\frac{x_{n+s}^{2}+x_{n+1}^{2}}{x_{n+s}^{3}}
&\cdots
&-\frac{x_{n+s-1}x_{n+1}}{x_{n+s}^{3}} \\
\vdots &   & \vdots&\vdots&  &\vdots      \\
-\frac{x_{n+s-1}g\frac{\partial g}{\partial x_{1}}}{x_{n+s}^{3}} &
\cdots   &
-\frac{x_{n+s-1}g\frac{\partial g}{\partial x_{n}}}{x_{n+s}^{3}} &
-\frac{x_{n+s-1}x_{n+1}}{x_{n+s}^{3}}
&\cdots&-\frac{x_{n+s}^{2}+x_{n+s-1}^{2}}{x_{n+s}^{3}}
\\
\end{array}
\right).
$$
The determinant of this matrix equals, up to a sign, to
\begin{equation}\label{determinant2}
\frac{g^{n}}{x_{n+s}^{n+3(s-1)}}
\det\left(
\begin{array}{ccc}
\frac{\partial^{2}g}{\partial x_{1}^{2}} &\dots&   \frac{\partial^{2}
g}{\partial x_{1}\partial x_{n}}
\\
\dots &  &\vdots   \\
\frac{\partial^{2} g}{\partial
x_{1}\partial x_{n}}  & \dots&
\frac{\partial^{2}g}{\partial x_{n}^{2}}
\end{array}
\right)
\det\left(\begin{array}{ccc}
x_{n+s}^{2}+x_{n+1}^{2} &\dots&x_{n+s-1}x_{n+1}   \\
\vdots&   &\vdots    \\
x_{n+s-1}x_{n+1}& \dots& x_{n+s}^{2}+x_{n+s-1}^{2}
\end{array}
\right)
\end{equation}
It is left to evaluate the second determinant. To that end we use a
well-known matrix determinant formula: For any dimension $m$ and $y
\in \R^m$,
\begin{equation}
 \det(Id + y \otimes y) = 1 + \| y \|^2 \label{eq_1051} \end{equation} where $y \otimes y$ is
the matrix whose $(y_i y_j)_{i,j=1,\ldots,n}$.
Consequently, for
the second determinant in (\ref{determinant2}) we have
$$
\det\left(\begin{array}{ccc}
x_{n+s}^{2}+x_{n+1}^{2} &\dots&x_{n+s-1}x_{n+1}   \\
\vdots&   &\vdots    \\
x_{n+s-1}x_{n+1}& \dots& x_{n+s}^{2}+x_{n+s-1}^{2}
\end{array}
\right) =\left(\sum_{i=n+1}^{n+s} \frac{x_{i}^{2}}{x_{n+s}^2}
\right)x_{n+s}^{2s} =g^{2}x_{n+s}^{2(s-2)}.
$$
Therefore we get for the expression (\ref{determinant2})
$$
\frac{g^{n+2}}{x_{n+s}^{n+s+1}}
\det\left(
\begin{array}{ccc}
\frac{\partial^{2}g}{\partial x_{1}^{2}} &\dots&   \frac{\partial^{2}
g}{\partial x_{1}\partial x_{n}}
\\
\vdots &  &\vdots   \\
\frac{\partial^{2} g}{\partial
x_{1}\partial x_{n}}  & \dots&
\frac{\partial^{2}g}{\partial x_{n}^{2}}
\end{array}
\right).
$$
Moreover
\begin{eqnarray*}
1+\sum_{i=1}^{n+s-1}\left|\frac{\partial x_{n+s}}{\partial x_{i}}
\right|^{2} &=&1+\sum_{i=1}^{n}\left|\frac{g\frac{\partial
g}{\partial x_{i}}}{x_{n+s}} \right|^{2}
+\sum_{i=n+1}^{n+s-1}\left|\frac{x_{i}}{x_{n+s}}\right|^{2}
=\sum_{i=1}^{n}\left|\frac{g\frac{\partial g}{\partial
x_{i}}}{x_{n+s}} \right|^{2}
+\left|\frac{g}{x_{n+s}}\right|^{2}  \\
&=&\left|\frac{g}{x_{n+s}}\right|^{2}
\left(1+\sum_{i=1}^{n}\left|\frac{\partial g}{\partial x_{i}}
\right|^{2}\right).
\end{eqnarray*}
Therefore, we get by (\ref{curv}) for the curvature
\begin{eqnarray*}
\kappa(z)
&=&\frac{\frac{g^{n+2}}{x_{n+s}^{n+s+1}}
\det\left(
\begin{array}{ccc}
\frac{\partial^{2}g}{\partial x_{1}^{2}} &\dots&   \frac{\partial^{2}
g}{\partial x_{1}\partial x_{n}}
\\
\vdots &  &\vdots   \\
\frac{\partial^{2} g}{\partial
x_{1}\partial x_{n}}  & \cdots&
\frac{\partial^{2}g}{\partial x_{n}^{2}}
\end{array}
\right)}
{\left(\left|\frac{g}{x_{n+2}}\right|^{2}
\left(1+\sum_{i=1}^{n}\left|\frac{\partial g}{\partial x_{i}}
\right|^{2}\right)\right)^{\frac{n+s+1}{2}}}
=\frac{
\det\left(
\begin{array}{ccc}
\frac{\partial^{2}g}{\partial x_{1}^{2}} &\dots&   \frac{\partial^{2}
g}{\partial x_{1}\partial x_{n}}
\\
\vdots &  &\vdots   \\
\frac{\partial^{2} g}{\partial
x_{1}\partial x_{n}}  & \dots&
\frac{\partial^{2}g}{\partial x_{n}^{2}}
\end{array}
\right)}
{g^{s-1}
\left(1+\sum_{i=1}^{n}\left|\frac{\partial g}{\partial x_{i}}
\right|^{2}\right)^{\frac{n+s+1}{2}}}\\
&=& \frac{ \mbox{det} \left(\Hess \ f^\frac{1}{s}\right)}{f^\frac{s-1}{s} \left(1 +\|\nabla f^\frac{1}{s}\|^2\right)^\frac{n+s+1}{2}}.
\end{eqnarray*}
This completes the proof of Lemma \ref{normal+kappa}.
\end{proof}

\begin{proof}[Proof of Proposition
\ref{prop.s-affine-sa}]
Denote by $\tilde{\partial} K_s(f)$ the collection of all
points $(x_1,\ldots,x_{n+s}) \in \partial K_s(f)$ such that
$(x_1,\ldots,x_n) \in \interior(\Supp(f))$. Since there is no contribution to the integral of $as_1\left(K_s(f)\right)$ from $\partial K_s(f) \setminus \overline{\tilde{\partial} K_s(f)}$
(since   the Gauss curvature vanishes on the part with full dimension, if exists)
clearly
$$ as_1^{(s)}(f) = as_1\left(K_s(f)\right)= \int_{\partial K_s(f)} \kappa_{K_s(f)}
^\frac{1}{n+s+1} \ d\mu_{K_s(f)} = \int_{\tilde{\partial} K_s(f)}
\kappa_{K_s(f)} ^\frac{1}{n+s+1} \ d\mu_{K_s(f)}. $$
 By  Lemma
\ref{normal+kappa}
\begin{eqnarray} \label{as}
as_1^{(s)}(f)&=&  \int_{\tilde{\partial} K_s(f)}\frac{\left( \det\left(\Hess (f^\frac{1}{s})\right) \right)^\frac{1}{n+s+1}}{\left(1 + \|\nabla f^\frac{1}{s}\| \right)^\frac{1}{2}}\  f ^{-\frac{s-1}{s(n+s+1)}} \   d\mu_{K_s(f)} \nonumber \\
&=  &  2\ \int_{\R^{n+s-1}} f^\frac{1}{s}\ \left(
\frac{\det\left( \Hess(f^\frac{1}{s})\right)} {f^
\frac{s-1}{s} } \right)^\frac{1}{n+s+1} \frac{ dx_1 \dots dx_{n+s-1}
}{ |x_{n+s}|}
\end{eqnarray}
where $f$ is evaluated, of course, at $(x_1,\ldots,x_n)$. The last
equality follows as the boundary of $K_s(f)$ consists of two, ``positive'' and ``negative'',  parts. For $s=1$,  we get \[ 2
\int_{\R^{n}} \left(\det\left(
 \Hess f \right)
 \right)^\frac{1}{n+2}
dx_1 \dots dx_{n}, \]  hence $c_1=2$.
 For $s >1$,
\begin{eqnarray*}
\int_{\R^{s-1}}\frac{ dx_{n+1} \dots dx_{n+s-1} }{
|x_{n+s}|}&=& \int_{\R^{s-1}}  f^{-\frac{1}{s}}\
\bigg(1-\sum_{i=n+1}^{n+s-1} \left(\frac{x_i}{f^\frac{1}{s}}
\right)^2 \bigg)^{-\frac{1}{2}} \ dx_{n+1} \dots dx_{n+s-1}\\
&=&
\int_{\sum_{i=n+1}^{n+s-1}  y_i^2 \leq 1}
\frac{f^\frac{s-1}{s}}{f^\frac{1}{s}}\ \bigg(1-\sum_{i=n+1}^{n+s-1}
y_i^2 \bigg)^{-\frac{1}{2}} \ dy_{n+1} \dots dy_{n+s-1}\\
&=&
\frac{f^\frac{s-1}{s}}{f^\frac{1}{s}}\  (s-1) \mbox{vol}_{s-1}\left(B_2^{s-1}\right)  \
\int_0^1 \frac{r^{s-2} dr}{(1-r^2)^\frac{1}{2}}\\
&=&
\frac{f^\frac{s-1}{s}}{f^\frac{1}{s}}\  (s-1) \mbox{vol}_{s-1}\left(B_2^{s-1}\right)
\frac{1}{2}B\left(\frac{s-1}{2}, \frac{1}{2}\right).
\end{eqnarray*}
Thus (\ref{as}) becomes
$$
as_1^{(s)}(f)=(s-1)\mbox{vol}_{s-1}\left(B_2^{s-1}\right) B\left(\frac{s-1}{2}, \frac{1}{2}\right) \
\int_{\R^{n}}  f^\frac{s-1}{s}\ \left( \frac{
\mbox{det}\left(\mbox{Hess}(f^\frac{1}{s})\right)} {f^ \frac{s-1}{s} }
\right)^\frac{1}{n+s+1} \ dx,
$$
and the proof of Proposition \ref{prop.s-affine-sa} is  complete.
\end{proof}

With the formula for $as_1^{(s)}(f)$ in hand, we may use the affine
isoperimetric inequality for convex bodies to obtain the following corollary.
\begin{cor} \label{cor:s-ineq}
For all $s \in \mathbb{N}$ and for all $f \in Conc_s^{(2)}(\R^n)$ we have
$$
 \int_{\supp (f)}
\left| \det \left( \Hess f^\frac{1}{s} \right)\right|^\frac{1}{n+s+1}
f^{\frac{s-1}{s}\left( \frac{n+s}{n+s+1}\right)} \ dx \leq d(n,s) \left(\int_{\supp (f)}
f dx\right)^\frac{n+s-1}{n+s+1},
$$
where
$$d(n,s)=\pi^\frac{n}{n+s+1}
\left(\frac{n+s}{s}\right)^\frac{n+s-1}{n+s+1}
\left(\frac{\Gamma(\frac{s}{2})}{\Gamma(\frac{n+s}{2})}\right)^\frac{2}{n+s+1}.
$$
Equality holds if and only if $f=(a + \langle b, x \rangle - \langle
A x, x \rangle)_+^{s/2}$ for $a \in \R, b \in \R^n$ and a
positive-definite matrix $A$.
\end{cor}
\begin{proof}[Proof of Corollary \ref{cor:s-ineq}] The affine isoperimetric
inequality for convex bodies $K$ in $\R^n$ (see, e.g.,
\cite{Petty}) says that
\begin{equation}\label{aii}
\frac{as_1(K)}{as_1(B^n_2)} \leq
\left(\frac{\mbox{vol}_n\left(K\right)}{\mbox{vol}_n\left(B^n_2\right)}\right)^\frac{n-1}{n+1},
\end{equation}
with equality if and only if $K$ is an ellipsoid. We apply
(\ref{aii}) to $K_s(f) \subset \R^{n+s}$ and get
\begin{eqnarray*}
 \frac{as_1^{(s)}(f)}{as_1^{(s)}(g_s)} &=&
\frac{as_1\left(K_s(f)\right)}{as_1\left(K_s(g_s)\right)} \\
&=& \frac{c_s}{as_1\left(B^{n+s}_2\right)}
\int_{\supp (f)}
\bigg(\mbox{det}\big(\frac{\partial^2 f^\frac{1}{s}}{\partial x_i \partial x_j}\big)_{i,j =1,\ldots,n}\bigg)^\frac{1}{n+s+1} \ f^{\frac{(s-1)(n+s)}{s(n+s+1)}}  dx \\
& \leq & \left(\frac{
\mbox{vol}_s\left(B^s_2\right) \   \int_{\supp (f)}f dx}{\mbox{vol}_{n+s}\left(B^{n+s}_2\right)}\right)^\frac{n+s-1}{n+s+1},
\end{eqnarray*}
with equality if and only if $f(x)=(a + \langle b, x \rangle -
\langle A x, x \rangle)_+^{s/2}$ for $a \in \R, b \in \R^n$ and a
positive-definite matrix $A$. This is rewritten as
$$
\int_{\supp (f)} \left(\det\left(\Hess (
f^\frac{1}{s})\right)\right)^\frac{1}{n+s+1} \
f^{\frac{(s-1)(n+s)}{s(n+s+1)}} dx \leq d(n,s) \
\left(\int_{\supp (f)} f dx \right)^\frac{n+s-1}{n+s+1},
$$ where
\begin{eqnarray*} \nonumber
d(n,s)&=&\frac{(n+s)\ \mbox{vol}_{n+s}\left(B^{n+s}_2\right)}{c_s}
\left(\frac{\mbox{vol}_{s}\left(B^{s}_2\right)}{\mbox{vol}_{n+s}\left(B^{n+s}_2\right)}\right)^\frac{n+s-1}{n+s+1} \\
&=&\pi^\frac{n}{n+s+1}\left(\frac{n+s}{s}\right)^\frac{n+s-1}{n+s+1}\
\left(\frac{\Gamma(\frac{s}{2})}{\Gamma(\frac{n+s}{2})}\right)^\frac{2}{n+s+1}.
\end{eqnarray*}
\end{proof}

It follows immediately from the definition and from
Proposition \ref{prop.s-affine-sa}, that $as_1^{(s) }(f) $ is affine
invariant and that it is a valuation:
\begin{cor}\label{prop}
Let $s \in \mathbb{N}$ and let  $f \in Conc_s(\R^n)  \cap
C^2\left(\supp (f)\right)$. \vskip 2mm \noindent (i) For all linear maps $A: \R^n \rightarrow \R^n$ with
$\det A \neq 0$, and for all $\lambda \in \R$, we have
\[ as_1^{(s)}((\lambda f) \circ A) = \frac{\lambda^\frac{n+s-1}{n+s+1}}{|\det A|}as_1^{(s)}(f).\]
In particular, if
$|\det A|=1$,
$$
as_1^{(s)}(f \circ A) = as_1^{(s)}(f).
$$
\noindent
(ii) $as_1^{(s)}$ is a ``valuation'': If $\max
(f_1,f_2)$ is $s$-concave,  then
\[ as_1^{(s)}(f_1)+ as_1^{(s)}(f_2) =  as_1^{(s)}(\max (f_1, f_2)) +
as_1^{(s)}(\min(f_1, f_2))\]
\end{cor}
\begin{proof}[Proof of Corollary \ref{prop}]
(i) By Proposition \ref{prop.s-affine-sa},
\begin{eqnarray*}
&&as_1^{(s)}((\lambda f) \circ A )\\
&&=
c_s\  \int_{\supp ( f\circ A)}
\left|\det\left( \Hess \left( (\lambda f) \circ A\right)^\frac{1}{s}
\right) \right|^\frac{1}{n+s+1} \
\lambda^{\frac{(s-1)(n+s)}{s(n+s+1)}} f(Ax)^{\frac{(s-1)(n+s)}{s(n+s+1)} }\ dx \\
&&= c_s\  \frac{ \lambda^{\frac{n+s-1}{n+s+1}}}{|\det A|} \ \int_{\supp (f)}
\left|\det\left( \Hess (f ^\frac{1}{s}
)\right)  \right|^\frac{1}{n+s+1} \
 f^{\frac{(s-1)(n+s)}{s(n+s+1)} }\ dy\\
&& = \frac{ \lambda^{\frac{n+s-1}{n+s+1}}}{|\det A|} as_1^{(s)}( f ).
\end{eqnarray*}
\noindent (ii) By  (\ref{def:aff}) and since the affine
surface area for convex bodies is a valuation \cite{Schuett1},
\begin{eqnarray*}
as_1^{(s)}(f_1) + as_1^{(s)}(f_2) &=& as_1\left(K_s(f_1)\right) +  as_1\left(K_s(f_2)\right) \\
&= & as_1\left(K_s(f_1) \cup K_s(f_2) \right) +  as_1\left(K_s(f_1) \cap K_s(f_2) \right) \\
&= & as_1^{(s)}(\max (f_1, f_2)) + as_1^{(s)}(\min(f_1, f_2)),
\end{eqnarray*}
provided that $K_s(f_1) \cup K_s(f_2)$ is convex.
\end{proof}

\section{$\log$-concave functions}

We would like  to obtain an inequality corresponding to the one of
Corollary \ref{cor:s-ineq} not only for $s$-concave functions but,
more generally, for log-concave functions on $\R^n$, which are the natural functional extension of convex bodies. The union of all classes of
$s$ concave functions over all $s$ is dense within log-concave functions in many natural topologies.

Note that if a function $f$ is
$s_0$-concave for some $s_0$, then it is $s$-concave for all $s\ge
s_0$. Therefore,  by  Corollary \ref{cor:s-ineq}, we get that  for
any $s_0 \in \N$ and any $f \in Conc_{s_0}(\R^n)  \cap
C^2\left(\supp (f )\right)$we have  for all  $s\ge s_0$
\begin{eqnarray*}
 \int_{\supp (f)} f^{\frac{(s-1)(n+s)}{s(n+s+1)}}
\left|\det\left(\Hess (f^\frac{1}{s})\right)\right|^\frac{1}{n+s+1}  dx \leq d(n,s) \left(\int_{\supp (f )}
f dx\right)^\frac{n+s-1}{n+s+1}.
\end{eqnarray*}
Taking the limit as $s\to \infty$ one sees that the limit on both sides is simply
$\int_{\supp (f )} f dx$, so that one does not get an interesting
inequality. However, we may take the derivative at
$s=+\infty$ as in \cite{Klar} (the details are given in the proof
below), and doing so, we obtain the inequality of Theorem \ref{thm1}.

\medskip Before we present the proof of Theorem  \ref{thm1}, we give
an example  in which both sides are computable. The computation is
straightforward and left for the interested reader.
\begin{exam}\label{exam1}
Let $p>1$ and $f:\mathbb R^{n}\rightarrow\mathbb R$ be given by $f(x)=e^{-\sum_{i=1}^n |x_i|^p}$. Then
$$
\int_{\mathbb R^{n}} f \ \ln\bigg(\det\left(
\Hess\left(-\ln f\right)\right)\bigg)  dx = n \ \left(
\frac{2}{p} \ \Gamma\left(\frac{1}{p}\right) \right)^n  \left( \ln
\big(p(p-1)\big) + (p-2) \
\frac{\Gamma^{\prime}(\frac{1}{p})}{\Gamma(\frac{1}{p})}\right)
$$
and
$$
2 \bigg[  \operatorname{Ent}(f) +  \|f\|_{L^1( dx)} \ln(2 \pi e)^\frac{n}{2} \bigg]
=n \  \left( \frac{2}{p} \ \Gamma\left(\frac{1}{p}\right) \right)^n  \left( \ln \bigg(\frac{\pi e}{2 \Gamma(1+\frac{1}{p})^2}\bigg) - \frac{2}{p}\right).
$$
Both expressions are equal when  $p=2$.
\end{exam}

\vskip 3mm
\begin{proof}[Proof of Theorem  \ref{thm1}:]
One is given a function $f$ which is log-concave and $C^2$-smooth in
the interior of its support. In order to apply Corollary \ref{cor:s-ineq}, we
modify $f$ slightly as follows: For $\eps > 0$, set
$$ f_{\eps}(x) = f(x) \exp(-\eps \| x \|^2) \chi_{\{f \geq
\varepsilon\}}(x) \quad \quad \quad (x \in \R^n). $$ By a standard
compactness argument, every log-concave function with compact support is $s_0$-concave for some $s_0$. Hence there exists $s_0 > 0$ such that
$f_\varepsilon$ is $s$-concave for all  $s \geq s_0$ and thus (\ref{s-ineq1})
holds for $f_\varepsilon$ and any $s \geq s_0$. We expand the left hand side and
the right hand side of the inequality in Corollary \ref{cor:s-ineq}
in terms of $\frac{1}{s}$. We have
$$
\frac{\partial^2 f_{\eps}^\frac{1}{s}}{\partial x_i \partial
x_j}=\frac{1}{s} \frac{\partial}{\partial
x_j}\left(f_{\eps}^{\frac{1}{s}-1} \frac{ \partial
f_{\eps}}{\partial x_i}\right) = \frac{f_{\eps}^{\frac{1}{s}-2}}{s}
\left( f_{\eps} \frac{\partial^2 f_{\eps}}{\partial x_i \partial
x_j}  - \frac{\partial f_{\eps}}{
\partial x_j} \frac{\partial f_{\eps}}{ \partial x_i}  + \frac{1}{s} \
\frac{\partial f_{\eps}}{
\partial x_j} \frac{\partial f_{\eps}}{ \partial x_i} \right)
$$
Thus
\begin{eqnarray*}
\Hess(f_{\eps}^{1/s}) &=& \frac{f_{\eps}^\frac{1}{s}}{s}  \left(
\frac{f_{\eps} \Hess(f_{\eps}) - \nabla f_{\eps} \otimes \nabla
f_{\eps} + \frac{1}{s}
\nabla f_{\eps} \otimes \nabla f_{\eps} }{f_{\eps}^2} \right) \\
&=&\frac{f_{\eps}^\frac{1}{s}}{s}  \ \left( \Hess \left(\ln
f_{\eps}\right) + \frac{1}{s} \frac{\nabla f_{\eps} \otimes \nabla
f_{\eps} }{f_{\eps}^2} \right)
\end{eqnarray*}
and hence
$$
\mbox{det}\left(-\frac{\partial^2 f_{\eps}^\frac{1}{s}}{\partial x_i
\partial x_j}\right)_{i,j =1,\ldots,n} = \frac{f_{\eps}^\frac{n}{s}}{s^n}  \
\mbox{det} \left( -\left( \Hess \left(\ln f_{\eps}\right) +
\frac{1}{s} \frac{\nabla f_{\eps} \otimes \nabla f_{\eps}
}{f_{\eps}^2} \right) \right).
$$
Thus the inequality of Corollary \ref{cor:s-ineq} is equivalent to
\begin{eqnarray} \label{s-ineq1}
&&\hskip -15mm \int_{\supp  (f_{\eps})} \left| \mbox{det} \left(
-\left( \Hess \left(\ln  f_{\eps}\right) + \frac{1}{s} \frac{\nabla
f_{\eps} \otimes \nabla f_{\eps} }{f_{\eps}^2} \right)
\right)\right|^\frac{1}{n+s+1} \ f_{\eps}^\frac{n+s-1}{n+s+1} dx
 \nonumber \\
&&\hskip 40mm \leq d(n,s) \ s ^\frac{n}{n+s+1}  \ \left(\int_{\supp
(f_{\eps})} f_{\eps} dx\right)^\frac{n+s-1}{n+s+1}.
\end{eqnarray}
Applying again the formula for the determinant of a rank-one
perturbation of a matrix, we have
\begin{eqnarray}\label{expand:det}
&& \det \left( -\left( \Hess \left(\ln f_{\eps}\right) + \frac{1}{s}
\frac{\nabla f_{\eps} \otimes \nabla
f_{\eps} }{f_{\eps}^2} \right)  \right)  \nonumber \\
&& =\det \left( - \Hess \left(\ln  f_{\eps}\right)  \right) \left[ 1
+s^{-2} f_{\eps}^{-2} \langle \left( \Hess \ln f_{\eps} \right)
^{-1} \nabla f_{\eps}, \nabla f_{\eps} \rangle  \right]
\nonumber \\
&& =\det \left( - \Hess \left(\ln  f_{\eps}\right)  \right) + s^{-2}
\alpha_{\eps}(x),
\end{eqnarray}
where, for a fixed $\eps$, the function $\alpha_{\eps}(x)$ is defined by \eqref{expand:det} and is clearly
bounded on the interior of the support of $f_{\eps}$.  We write, for
the left hand side of (\ref{s-ineq1}),
$$
f_\varepsilon^\frac{n+s-1}{n+s+1}= f_\varepsilon \ \left(f_\varepsilon^{-2}\right)^\frac{1}{n+s+1}
$$
and on the right hand side
$$
\left(\int_{\supp  (f_\varepsilon)}
f_\varepsilon dx\right)^\frac{n+s-1}{n+s+1}= \left(\int_{\supp  (f_\varepsilon)} f_\varepsilon dx\right) \left( \int_{\supp  (f_\varepsilon)} f_\varepsilon dx\right)^\frac{-2}{n+s+1}.
$$
Moreover,
\begin{eqnarray*}
d(n,s) \ s^\frac{n}{n+s+1} &=& \left(s\pi \right)^\frac{n}{n+s+1}\left(\frac{n+s}{s}\right)^\frac{n+s-1}{n+s+1}\
\left(\frac{\Gamma(\frac{s}{2})}{\Gamma(\frac{n+s}{2})}\right)^\frac{2}{n+s+1} \\
&\leq& \left(2 \pi e\right)^\frac{n}{n+s+1} \left(1 +\frac{1}{3s}\right)^\frac{2}{n+s+1},
\end{eqnarray*}
where we have used that
for $x \rightarrow \infty$,
\begin{eqnarray*}\label{gamma}
\Gamma(x)= \sqrt{2\pi} \ x^{x-\frac{1}{2}}\  e^{-x}\ \left[1+
\frac{1}{12x}  + \frac{1}{288x^2}  \pm o(x^{-2})\right],
\end{eqnarray*}
and we make the legitimate assumption that $s$ is sufficiently
large.
 Thus, together with (\ref{expand:det}), it follows from
(\ref{s-ineq1}) that
\begin{eqnarray} \label{s-ineq2}
&&\int_{\supp  (f_\varepsilon)} f_\varepsilon\left|
f_\varepsilon^{-2}\left(\det \left( - \Hess \left(\ln
f_\varepsilon\right) \right)
+s^{-2} \alpha_{\eps}(x) \right)\right|^\frac{1}{n+s+1} dx \nonumber \\
&&\leq \left(\int_{\supp  (f_\varepsilon)} f_\varepsilon dx \right) \left(\left(1+\frac{1}{3s}\right)^2 (2 \pi e)^n \left(\int_{\supp  (f_\varepsilon)} f_\varepsilon dx \right)^{-2} \right)^\frac{1}{n+s+1}.
\end{eqnarray}
We estimate the left hand side of (\ref{s-ineq2}) from below by
\begin{eqnarray}\label{s-ineq3}
&&\int_{\supp  (f_\varepsilon)} f_\varepsilon\left|
f_\varepsilon^{-2}\left(\det \left( -\left( \Hess  \ln
f_\varepsilon\right)  \right)
+ s^{-2} \alpha_{\eps}(x) \right)\right|^\frac{1}{n+s+1} dx \nonumber \\
&&=\int_{\supp  (f_\varepsilon)}f_\varepsilon \exp\bigg(\frac{1}{n+s+1}\ln\bigg| f_\varepsilon^{-2}\bigg(\det
\left( -\left( \Hess
 (\ln  f_\varepsilon\right)    \right)
 + s^{-2} \alpha_\eps(x) \bigg) \bigg|\bigg) dx \nonumber\\
&& \geq \int_{\supp  (f_\varepsilon)} f_\varepsilon
\bigg( 1 + \frac{1}{n+s+1} \ln
\bigg| f_\varepsilon^{-2}\bigg(\det
\left( -\left( \Hess \ln  f_\varepsilon\right)  \right)
 +s^{-2} \alpha_{\eps}(x) \bigg)  \bigg|\bigg) dx.
\nonumber
\end{eqnarray}
We write the  right hand side of (\ref{s-ineq2})
\begin{eqnarray*}
&&\left(\int_{\supp  (f_\varepsilon)} f_\varepsilon dx \right) \left(\left(1+\frac{1}{3s}\right)^2 (2 \pi e)^n \left(\int_{\supp  (f_\varepsilon)} f_\varepsilon dx \right)^{-2} \right)^\frac{1}{n+s+1} \\
&&=
\left(\int_{\supp  (f_\varepsilon)} f_\varepsilon dx \right) \sum_{j=0}^\infty
\frac{1} {j! (n+s+1)^j}
\left(
\ln\left(\frac{\left(1+\frac{1}{3s}\right)^2 (2 \pi e)^n}{\left(\int_{\supp  (f_\varepsilon)} f_\varepsilon dx\right)^2}\right)
\right)^j.
\end{eqnarray*}
Therefore we get the following inequality
\begin{eqnarray}\label{s-ineq4}
 \int_{\supp  (f_\varepsilon)} f_\varepsilon
\bigg( 1 + \frac{1}{n+s+1} \ln
\bigg| f_\varepsilon^{-2}\bigg(\det
\left( -\left( \Hess \ln  f_\varepsilon\right) \right)+ s^{-2} \alpha_{\eps}(x) \bigg)  \bigg|\bigg) dx \nonumber \\
\leq
\left(\int_{\supp  (f_\varepsilon)} f_\varepsilon dx \right) \sum_{j=0}^\infty
\frac{1} {j! (n+s+1)^j}
\left(
\ln\left(\frac{\left(1+\frac{1}{3s}\right)^2 (2 \pi e)^n}{\left(\int_{\supp  (f_\varepsilon)} f_\varepsilon dx\right)^2}\right)
\right)^j.
\end{eqnarray}
We subtract the first order term $\int_{\supp  (f_\varepsilon)} f_\varepsilon dx$
from both side, multiply by $n+s+1$ and take the limit as $s\to \infty$. We get

\begin{eqnarray*}
&& \liminf_{s \to \infty} \int_{\supp  (f_\varepsilon)}
f_\varepsilon \bigg( \ln \bigg| f_\varepsilon^{-2}\bigg(\det
\left( -\left( \Hess \ln  f_\varepsilon\right)
\right)
+ s^{-2} \alpha_{\eps}(x)  \bigg)  \bigg|\bigg) dx \nonumber \\
&& \leq \limsup_{s \to \infty} \int_{\supp  (f_\varepsilon)}
f_\varepsilon \bigg( \ln \bigg| f_\varepsilon^{-2}\bigg(\det
\left( -\left( \Hess \ln  f_\varepsilon\right)
\right)
+ s^{-2} \alpha_{\eps}(x) \bigg)  \bigg|\bigg) dx \nonumber \\
&& \leq
\left(\int_{\supp  (f_\varepsilon)} f_\varepsilon dx \right) \limsup_{s \to \infty} \sum_{j=1}^\infty
\frac{1} {j! (n+s+1)^{j-1}}
\left(
\ln\left(\frac{\left(1+\frac{1}{3s}\right)^2 (2 \pi e)^n}{\left(\int_{\supp  (f_\varepsilon)} f_\varepsilon dx\right)^2}\right)
\right)^j \nonumber\\
&&=
\left(\int_{\supp  (f_\varepsilon)} f_\varepsilon dx \right) \ \ln\left(\frac{ (2 \pi e)^n}{\left(\int_{\supp  (f_\varepsilon)} f_\varepsilon dx\right)^2}\right).
\end{eqnarray*}

In the interior of the support of $f_{\eps}$, the Hessian of
$\nabla^2 (\ln f_{\eps})$ is greater than $\eps Id$, hence we can
apply Fatou's lemma on the left hand side to get
\begin{eqnarray*}
&&  \int_{\supp  (f_\varepsilon)} \liminf_{s \to \infty}
f_\varepsilon \bigg(  \ln \bigg| f_\varepsilon^{-2}\bigg(\det
\left( -\left( \Hess
\left(\ln  f_\varepsilon\right) \right)_{i,j =1,\ldots,n}
\right)
+ s^{-2} \alpha_{\eps}(x) \bigg)  \bigg|\bigg) dx \nonumber \\
&&\leq
\left(\int_{\supp  (f_\varepsilon)} f_\varepsilon dx \right) \ \ln\left(\frac{ (2 \pi e)^n}{\left(\int_{\supp  (f_\varepsilon)} f_\varepsilon dx\right)^2}\right),
\end{eqnarray*}
which simplifies to
\begin{eqnarray}\label{s-ineq5}
&&  \int_{\supp  (f_\varepsilon)}  f_\varepsilon
\bigg(  \ln  \bigg(\det \left( -\left(
\Hess \ln
f_\varepsilon\right)   \right)
 \bigg)   \bigg) dx \nonumber \\
&& \leq \left(\int_{\supp  (f_\varepsilon)} f_\varepsilon dx \right) \ \ln\left(\frac{ (2 \pi e)^n}{\left(\int_{\supp  (f_\varepsilon)} f_\varepsilon dx\right)^2}\right)
+ 2 \int_{\supp  (f_\varepsilon)}  f_\varepsilon
  \ln   f_\varepsilon
  dx
.
\end{eqnarray}
Now we pass to the limit $\varepsilon \to 0$ on both sides of
(\ref{s-ineq5}). We deal with each of the three terms separately.
For the first term, since $-\ln f_{\eps} = -\ln f + \eps \|\cdot
\|^2/2$, we have
\[ \int_{\{f\ge \varepsilon\} }  f_\eps
\bigg(  \ln  \bigg(\det \left(-\Hess (\ln f)  + \eps Id \right)
\bigg) dx \ge \int_{\{f\ge \varepsilon\} }  f_\eps \bigg(  \ln
\bigg(\det \left(-\Hess \ln f\right)  \bigg) dx. \] Since the
integral $  f
  \ln   (\det
 (\Hess \ln f )  )$ is assumed to belong to $L_1$
 and   $f_\varepsilon$ increases monotonously to $f$ as $\varepsilon \to 0$,
  the integrand is bounded by $ f \left|\ln   (\det
 (\Hess \ln f )  )\right|$ and by the dominated convergence theorem
\[ \lim_{\eps\to 0} \int_{\{f\ge \varepsilon\} }  f_\eps
\bigg(  \ln  \bigg(\det \left(-\Hess \ln f\right)  \bigg) dx =
\int_{\supp f}  f \bigg(  \ln  \bigg(\det \left(-\Hess \ln f\right)
\bigg) dx.\] Similarly, monotone convergence theorem ensures that
$$
\lim_{\varepsilon \to 0} \left(\int_{\supp  (f_\varepsilon)} f_\varepsilon dx \right) \ \ln\left(\frac{ (2 \pi e)^n}{\left(\int_{\supp  (f_\varepsilon)} f_\varepsilon dx\right)^2}\right)= \left(\int_{\supp  (f)} fdx \right) \ \ln\left(\frac{ (2 \pi e)^n}{\left(\int_{\supp  (f)} f dx\right)^2}\right).
$$
We are left with showing that for the entropy function
\[ \lim_{\eps\to 0}\int  f_\varepsilon
  \ln   f_\varepsilon
    = \int f \ln f. \]
This is straightforward from the definition of $f_\eps$ and the assumptions on $f$, as
\[   f_\varepsilon
  \ln   f_\varepsilon  =   \left( e^{-\eps \|x\|^2/2}f\ln f  +
  \eps f e^{-\eps \|x\|^2/2}  \|x\|^2/2\right) \chi_{\{f\ge \eps\}}.\]
For the first term, apply again the dominated convergence theorem,
 and the second term disappears since the second moment of $f_{\eps}$ is bounded
 uniformly by the second moment of $f$. We end up with
\begin{eqnarray}\label{log-ineq5}
&&  \int_{\supp  (f)}  f
\bigg(  \ln  \bigg(\det \left( -\left(
\Hess \ln
f \right)   \right)
 \bigg)   \bigg) dx \nonumber \\
&& \leq \left(\int_{\supp  (f)} f \right) \ \ln\left(\frac{ (2 \pi e)^n}{\left(\int_{\supp  (f)} f dx\right)^2}\right)
+ 2 \int_{\supp  (f)}  f
  \ln   f.
\end{eqnarray}
This completes the proof of the main inequality.
The equality case is easily verified, and in particular follows from the affine invariance together with the computation in Example \ref{exam1}.
\end{proof}

\section{Linearization}\label{lin}

In this section we prove Corollary \ref{CorPoin}, be means of linearization
of our main inequality around its equality case. for convenience, we rewrite the
inequality of  Theorem \ref{thm1} in terms of a convex function
$\psi: \mathbb R^{n} \to \mathbb R$ such that
 $f=e^{-\psi}$.
We get
\begin{eqnarray}\label{mainineq55}
&& \hskip -10mm \int_{\mathbb R^{n}} e^{-\psi} \ln(\det(\Hess(\psi))) dx  \leq \nonumber\\
&& \hskip -10mm 2 \left\{ -\int_{\mathbb R^{n}} e^{-\psi} {\psi} dx-
\left(\int_{\mathbb R^{n}} e^{-\psi}dx\right) \ln\left( \int_{\mathbb R^{n}} e^{-\psi}dx \right)
+ \left( \int_{\mathbb R^{n}} e^{-\psi}dx \right)
 \ln(2 \pi e)^\frac{n}{2} \right\}.
\end{eqnarray}
Note that the support of $f$ is $\mathbb R^{n}$.
We then linearize
around the equality case $\psi(x) =   \|x\|^2/2$.

\begin{proof}[Proof of  Corollary \ref{CorPoin}] We first prove the corollary for functions with bounded
support.
Thus, let $\varphi$ be a twice continuously differentiable
function with bounded support and let $\psi(x)
= \|x\|^2/2 + \eps \varphi(x)$. Note that for sufficiently small
$\eps$ the function $\psi$ is convex. Therefore we can plug $\psi$ into inequality
(\ref{mainineq55}) and develop in
powers of $\eps$. We evaluate first the left hand
expression of (\ref{mainineq55}).
Since
$
\Hess(\psi)=I+\eps \varphi
$,
we obtain for the left hand side
$$
 \int_{{\mathbb R}^n} e^{-\|x\|^2/2 - \eps \varphi} \ln (\det (I
+ \eps \Hess \varphi)) dx.
$$
By Taylor's theorem this equals
$$
 \int_{{\mathbb R}^n} e^{-\|x\|^2/2} \left(1 - \eps \varphi +
\frac{\eps^2}{2} \varphi^2 \right) \cdot
 \ln (\det (I
+ \eps \Hess \varphi)) dx
+ O(\eps^3).
$$
For a matrix $A=(a_{i,j})_{i,j=1,\ldots,n}$, let $D(A) = \sum_{i
=1}^n\sum_{j\neq i}^n [a_{i,i}a_{j,j} - a_{i,j}^2]$. Note that each
$2\times 2$ minor is counted twice. Then
$$
\det(I+\eps \Hess \varphi)
=1+\eps \triangle \varphi
+\frac{\eps^{2}}{2}D(\Hess \varphi)+O(\eps^{3})
$$
where
$\triangle \varphi = \operatorname{tr}( \Hess \varphi)$
is the Laplacian of $\varphi$. Therefore
the left hand side equals
\begin{eqnarray*}
&& \int_{{\mathbb R}^n} e^{-\|x\|^2/2} \left(1 - \eps \varphi +
\frac{\eps^2}{2} \varphi^2 \right) \cdot
\left( \eps \triangle \varphi + \frac{\eps^2}{2}D(\Hess \varphi) -
\frac{\eps^2}{2} (\triangle \varphi)^2
 \right)dx
+ O(\eps^3) \\
&&= \eps \int_{{\mathbb R}^n} e^{-\|x\|^2/2} \triangle \varphi  dx+
\eps^2 \int_{{\mathbb R}^n} e^{-\|x\|^2/2} \left[ - \varphi \triangle
\varphi  + \frac{D(\Hess \varphi) - (\triangle \varphi)^2}{2}
\right]dx + O(\eps^3)
\\
&&= \eps \int_{\mathbb R^n} \left(\|x\|^2 - n\right) e^{-\|x\|^2/2}
\varphi + \eps^2 \int_{{\mathbb R}^n} e^{-\|x\|^2/2} \left[ - \varphi
\triangle \varphi - \frac{\|\Hess \varphi\|_2^2}{2} \right] +
O(\eps^3).
\end{eqnarray*}
The last equation follows by twice integration by parts.
\par
Now we evaluate the right hand side expression. First consider
\begin{eqnarray*}
\int_{\mathbb R^{n}}
e^{-\psi}dx
&=& \int_{\mathbb R^n}  e^{-\|x\|^2/2}dx - \eps \int_{\mathbb R^n}
e^{-\|x\|^2/2} \varphi dx+ \eps^2 \int_{\mathbb R^n}  e^{-\|x\|^2/2} \frac{\varphi^2}{2}dx
+O(\eps^3).
\end{eqnarray*}
Next,
\begin{eqnarray*}
-\int_{\mathbb R^{n}} e^{-\psi} \psi dx
&=& - \int_{\mathbb R^n}  e^{-\|x\|^2/2} \left( 1 - \eps \varphi + \eps^2
\frac{\varphi^2}{2} \right) \cdot \left( \frac{\|x\|^2}{2} + \eps
\varphi dx
\right) + O(\eps^3)  \\
&=&
-\int_{\mathbb R^n}  e^{-\|x\|^2/2}\frac{\|x\|^2}{2} +
\eps\left(\int_{{\mathbb R}^n} \varphi e^{-\|x\|^2/2}
\left(\frac{\|x\|^2}{2}-1\right)dx\right)\\
&&+\eps^2 \left( \int_{\mathbb R^n} \varphi^2 e^{-\|x\|^2/2}
\left(1-\frac{\|x\|^2}{4}\right)dx\right) + O(\eps^3)
\end{eqnarray*}
To treat $\left(\int_{\mathbb R^{n}} e^{-\psi}dx\right) \ln\left( \int_{\mathbb R^{n}} e^{-\psi}dx\right)$,
we consider the function $g(y) = y\ln y$, which we will  apply
to $\int e^{-\|x\|^2/2-\eps \varphi}$.
We obtain
\begin{eqnarray*}
&&\int_{\mathbb R^n} e^{-\|x\|^2/2
- \eps \varphi} dx \  \ln \left(\int_{\mathbb R^n} e^{-\|x\|^2/2 - \eps \varphi}dx\right)
\\
&&=\frac{n}{2}({2\pi})^{n/2}\ln({2\pi})
+\eps \left(   -\left(\frac{n}{2}\ln({2\pi})+1\right)\int_{\mathbb R^n} e^{-\|x\|^2/2} \varphi dx
\right)\\
&&+\eps^2\left(   \left(\frac{n}{2}\ln ({2\pi}) +1\right)\int_{\mathbb R^n} e^{-\|x\|^2/2}
\frac{\varphi^2}{2} dx+ \frac{1}{2({2\pi})^{n/2}}\left(\int_{\mathbb R^n} e^{-\|x\|^2/2}
\varphi  dx\right)^2 \right)
 + O(\eps^3)
 \end{eqnarray*}
Altogether,  the right hand side equals
\begin{eqnarray*}
&&2 \left\{ \int_{\mathbb R^n} e^{-\|x\|^2/2-\eps \varphi}
(-\|x\|^2/2-\eps\varphi)dx - \left(\int_{\mathbb R^n}
e^{-\|x\|^2/2-\eps\varphi}dx\right) \ln \left(\int_{\mathbb R^n}
e^{-\|x\|^2/2-\eps \varphi}dx \right)
\right\} \\
&&+ \left( \int_{\mathbb R^n} e^{-\|x\|^2/2-\eps \varphi}dx \right)
n\ln(2 \pi e)
\\&&=
-\int_{\mathbb R^n} e^{-\|x\|^2/2}{\|x\|^2}dx - n(2\pi)^{n/2}\ln(2\pi)
 +n \ln(2\pi e) \int_{\mathbb R^n} e^{-\|x\|^2/2}
dx\\
&&+ \eps\left\{ 2\left(\int_{\mathbb R^n}
 \varphi e^{-\|x\|^2/2}
(\frac{\|x\|^2}{2}-1)dx\right) - 2 \left(   -(\frac{n}{2}\ln({2\pi})+1)\int_{\mathbb R^n}
e^{-\|x\|^2/2} \varphi dx\right)\right. \\
&&\left.- n \ln (2\pi e)\int e^{-\|x\|^2/2}
\varphi\right\} \\
&&+\eps^2\left\{ \int_{{\mathbb R}^n}
\varphi^2
e^{-\|x\|^2/2} (1+\frac{n}{2}-\frac{1}{2}\|x\|^2)dx -
\frac{1}{({2\pi})^{n/2}}\left(\int_{{\mathbb R}^n} e^{-\|x\|^2/2} \varphi dx\right)^2\right\}
+O(\eps^3).
\end{eqnarray*}
Since
$$
\int_{{\mathbb R}^n} e^{-\|x\|^2/2}  dx
=({2\pi})^{n/2}
\hskip 10mm \mbox{and} \hskip 10mm
\int_{{\mathbb R}^n}\|x\|^{2} e^{-\|x\|^2/2}  dx
=n({2\pi})^{n/2},
$$
we get for the zeroth order term,
\begin{eqnarray*}
&&-\int_{\mathbb R^n} e^{-\|x\|^2/2}{\|x\|^2}dx - n(2\pi)^{n/2}\ln(2\pi)
 +n \ln(2\pi e) \int_{\mathbb R^n} e^{-\|x\|^2/2}
dx  \\
&&=-n(2\pi)^{n/2}
-n(2\pi)^{n/2}\ln(2\pi)
 +n \ln(2\pi e) (2\pi)^{n/2}=0.
\end{eqnarray*}
Therefore,  we get for the right hand side
\begin{eqnarray*}
&& \eps \int_{{\mathbb R}^n}\varphi
e^{-\|x\|^2/2}  {(\|x\|^2-n)} dx \\
&&+\eps^2 \left\{ \int_{{\mathbb R}^n}\varphi^2
e^{-\|x\|^2/2}
\left(\frac{n+2-\|x\|^2}{2}\right)dx - \frac{1}{({2\pi})^{n/2}}\left(\int_{{\mathbb R}^n} e^{-\|x\|^2/2}
\varphi dx\right)^2\right\}  +O(\eps^3).
\end{eqnarray*}
The coefficients of $\eps$ on the left and right hand side  are the same and we disacrd them. We   divide both sides by $\eps^{2}$ and take the limit
for $\eps\to0$. Then
\begin{eqnarray*}
&&\hskip -15mm \int_{{\mathbb R}^n} e^{-\|x\|^2/2} \left[ - \varphi \triangle \varphi -
\frac{\|\Hess \varphi\|_2^2}{2} \right]dx \\
&&\le   \int_{{\mathbb R}^n}\varphi^2
e^{-\|x\|^2/2} (\frac{n+2-\|x\|^2}{2}) dx-
\frac{1}{({2\pi})^{n/2}}\left(\int_{{\mathbb R}^n} e^{-\|x\|^2/2} \varphi
dx\right)^2.
\end{eqnarray*}
If we want the right hand side to include the variance, we may write
the inequality  as follows
\begin{eqnarray}\label{sohalt}
&&\hskip -15mm \int_{{\mathbb R}^n} e^{-\|x\|^2/2} \left[ - \varphi \triangle \varphi -
\frac{\|\Hess \varphi\|_2^2}{2} \right]dx \nonumber \\
&& \le
 \int_{{\mathbb R}^n}\varphi^2
e^{-\|x\|^2/2} \left(\frac{n-|x|^2}{2}\right)dx + ({2\pi})^{n/2}
\left[\int_{{\mathbb R}^n} \varphi^2 d\gamma_n
- \left(\int_{{\mathbb R}^n} \varphi d\gamma_n\right)^2\right]
\end{eqnarray}
Now we integrate on the right by parts twice, noting that
$(n-\|x\|^2)e^{-\|x\|^2/2} =
\triangle(e^{-\|x\|^2/2})$, so that the first term on the right hand
side is
\begin{eqnarray*} \int_{\mathbb R^n} e^{-\|x\|^2/2} \varphi^2 (\frac{n - \|x\|^2}{2}) dx
&=&-\frac{1}{2}\int_{\mathbb R^n} e^{-\|x\|^2/2} \triangle(\varphi^2) dx  \\
&=& -\int_{\mathbb R^n}
e^{-\|x\|^2/2} (\varphi\triangle \varphi + \|\grad \varphi\|^2)dx.
\end{eqnarray*}
We put that in (\ref{sohalt}) and   one gets
\[\int_{{\mathbb R}^n} e^{-|x|^2/2} \left[ \|\grad \varphi\|^2 -
\frac{\|\Hess \varphi\|_{HS}^2}{2}
\right] dx
\le  {({2\pi})^{n/2}}\left[ \int_{{\mathbb R}^n} \varphi^2 d\gamma_n - \left(\int_{{\mathbb R}^n}
\varphi d\gamma_n\right)^2\right], \]
which we  can rewrite  as
\[\int_{{\mathbb R}^n}    \|\grad \varphi\|^2 - \frac{\|\Hess \varphi\|_2^2}{2}
 d\gamma_n
\le  \int_{{\mathbb R}^n} \varphi^2 d\gamma_n - \left(\int_{{\mathbb R}^n}
\varphi d\gamma_n\right)^2.\]
Thus we have shown that the inequality holds for all twice continuously
differentiable functions $\varphi$ with bounded support.
One may extend it to
all twice continuously differentiable functions $\varphi\in L^{2}(\mathbb R^{n},\gamma_{n})$
with $\| \Hess \varphi\|_{HS} \in L^{2}(\mathbb R^{n},\gamma_{n})$ by a standard approximation argument, as follows.

Let $\chi_{k}$ be a twice continuously differentiable function
bounded between zero and one such that $\chi_{n}(x)=1$ for all
$\|x\|\leq k$ and $\chi_{n}(x)=0$ for all $\|x\|>k+1$. Then, for all
$k\in\mathbb N$
$$
\int_{\mathbb R^{n}}  \left[   \|\grad(\varphi\circ\chi_{k}) \|^2 - \frac{\| \Hess \varphi\|_{HS}^2}{2}
\right] d\gamma_{n}
\le  \int_{\mathbb R} (\varphi\circ\chi_{k})^2 d\gamma - \left(\int_{\mathbb R^{n}}(\varphi\circ\chi_{k})
d\gamma_{n}\right)^2,
$$
or, equivalently,
$$
\int_{\mathbb R^{n}} \|\grad(\varphi\circ\chi_{k})\|^2
 d\gamma +  \left(\int(\varphi\circ\chi_{k})
d\gamma_{n}\right)^2
\le  \int_{\mathbb R^{n}} (\varphi\circ\chi_{k})^2 d\gamma
+\int_{\mathbb R^{n}} \frac{\| \Hess \varphi\|_{HS}^2}{2}
d\gamma_{n}.
$$
It follows that
\begin{eqnarray*}
&&\hskip -15mm \liminf_{k\to\infty}\int_{\mathbb R^{n}} \|\grad(\varphi\circ\chi_{k})\|^2
 d\gamma_{n}
 +  \liminf_{k\to\infty}\left(\int(\varphi\circ\chi_{k})
d\gamma_{n}\right)^2   \\
&&\hskip 15mm \le \limsup_{k\to\infty} \int_{\mathbb R^{n}} (\varphi\circ\chi_{k})^2 d\gamma_{n}
+\limsup_{k\to\infty}\int_{\mathbb R^{n}} \frac{\| \Hess \varphi\|_{HS}^2}{2}
d\gamma_{n}.
\end{eqnarray*}
By Fatou's lemma and the dominated convergence theorem
\begin{eqnarray*}
&&\hskip -15mm\int_{\mathbb R^{n}} \liminf_{k\to\infty} \|\grad(\varphi\circ\chi_{k})\|^2
 d\gamma_{n}
 + \left(\int_{\mathbb R^{n}}\lim_{k\to\infty}(\varphi\circ\chi_{k})
d\gamma_{n}\right)^2   \\
&& \hskip 15mm \le \int_{\mathbb R} \lim_{k\to\infty}
(\varphi\circ\chi_{k})^2 d\gamma +\int_{\mathbb R}
\limsup_{k\to\infty} \frac{\|
\Hess \varphi\|_{HS}^2}{2} d\gamma,
\end{eqnarray*}
which gives
\begin{eqnarray*}
\int_{\mathbb R^{n}} \|\grad \varphi \|^2
 d\gamma
 + \left(\int_{\mathbb R^{n}}\varphi
d\gamma\right)^2
\le \int_{\mathbb R^{n}} \varphi^2 d\gamma_{n}
+\int_{\mathbb R^{n}} \frac{\| \Hess \varphi\|_{HS}^2}{2}
d\gamma_{n}.
\end{eqnarray*}
\end{proof}

\medskip An alternative, direct proof of Corollary \ref{CorPoin} may be given by
expanding $\vphi \in C^2(\R^n) \cap L^2(\R^n, \gamma_n)$ into
Hermite polynomials. That is, denote by $h_0(x),h_1(x),\ldots$ the
Hermite polynomials in one variable, normalized so that $\| h_i
\|_{L^2(\gamma_1)} = 1$ for all $i$. We may decompose
$$ \vphi = \sum_{i_1,\ldots,i_n = 0}^{\infty} a_{i_1,\ldots,i_n} \prod_{j=1}^n h_{i_j}(x_i) $$
where the convergence is in $L^2(\R^n, \gamma_n)$. Then the
right-hand side of (\ref{CorPoin1_}) equals
\begin{equation}
\label{eq_1054} \sum_{i_1,\ldots,i_n = 0 \atop{(i_1,\ldots, i_n)
\neq (0,\ldots,0)}}^{\infty} a_{i_1,\ldots,i_n}^2. \end{equation}
Using the identity $h_i^{\prime} =  \sqrt{i} \cdot h_{i-1}$, we see
that the left-hand side of (\ref{CorPoin1_}) is
\begin{equation}
 \int_{\mathbb R^{n}}  \left[ \left\|\grad \varphi\right\|^2_{} -
\frac{\|\Hess \varphi\|_{HS}^2}{2} \right] d\gamma_n =
\sum_{i_1,\ldots,i_n = 0}^{\infty} \left[ \frac{3}{2} \sum_{j=1}^n
i_j - \frac{1}{2} \left( \sum_{j=1}^n i_j \right)^2
 \right] a_{i_1,\ldots,i_n}^2. \label{eq_1110}
\end{equation}
We will use the simple fact that $x(3 - x) / 2 \leq 1$ for any
integer $x \geq 1$, for $x = \sum_{j=1}^n i_j$. Glancing at
(\ref{eq_1054}) with (\ref{eq_1110}) and using the aforementioned
simple fact, we deduce Corollary \ref{CorPoin}. We also see that
equality in (\ref{CorPoin1_}) holds if and only if $\vphi$ is a
polynomial of degree at most $2$, because $x (3-x) / 2 = 1$ only for
$x=1,2$.

The proof of Theorem \ref{gen-inv-poinc} is along the exact same lines,
using the all the derivatives are diagonalized by the Hermite polynomials with respect to the Gaussian measure, only that the inequality $x(3 - x) / 2 \leq 1$, which can be rewritten as $(x-1)(x-2)\ge 0$ for integers $x\ge 1$, is replaced by the more general inequality  $(x-1)(x-2)\cdots (x-j)\ge 0$ for integers $x\ge 1$, with equality if and only if $x\in \{1, \ldots, j\}$.

We remark that it is desirable to find an alternative,
direct proof of Theorem \ref{thm1}, which does not rely on the
affine isoperimetric inequality.

\medskip

{\bf Acknowledgement} We would like to thank Dario Cordero-Erausquin and Matthieu Fradelizi for helpful conversations.
Part of the  work was done during the authors stay at the Fields Institute, Toronto,  in the fall of 2010.  The paper was finished while
the last two named authors stayed at the  Institute for Mathematics and its Applications, University of Minnesota, in the fall of 2011. Thanks  go to both institutions for their hospitality.

\vskip 2mm
\noindent
Shiri Artstein-Avidan\\
{\small School of Mathematical Sciences}\\
{\small Tel-Aviv University }\\
{\small Tel-Aviv 69978,  Israel}\\
{\small \tt shiri@post.tau.ac.il} \\ \\
\noindent
\and
Bo'az Klartag\\
{\small School of Mathematical Sciences}\\
{\small Tel-Aviv University }\\
{\small Tel-Aviv 69978,  Israel}\\
{\small \tt klartagb@tau.ac.il} \\ \\
\noindent
\and
Carsten Sch\"utt\\
{\small Mathematisches Institut}\\
{\small Universit\"at Kiel}\\
{\small 24105 Kiel, Germany}\\
{\small \tt schuett@math.uni-kiel.de }\\ \\
\noindent
\and
Elisabeth Werner\\
{\small Department of Mathematics \ \ \ \ \ \ \ \ \ \ \ \ \ \ \ \ \ \ \ Universit\'{e} de Lille 1}\\
{\small Case Western Reserve University \ \ \ \ \ \ \ \ \ \ \ \ \ UFR de Math\'{e}matique }\\
{\small Cleveland, Ohio 44106, U. S. A. \ \ \ \ \ \ \ \ \ \ \ \ \ \ \ 59655 Villeneuve d'Ascq, France}\\
{\small \tt elisabeth.werner@case.edu}\\ \\

\end{document}